\documentclass[11pt]{article}
\usepackage{amssymb}
\usepackage{amssymb,amsthm}
\usepackage{amsmath}
\def\bc{\begin{center}}
\def\ec{\end{center}}

\def\s2c{\vskip 2cm}
\begin{document}
\numberwithin{equation}{section}
\newtheorem{theorem}{Theorem}[section]
\newtheorem{lemma}[theorem]{Lemma}
\newtheorem{remark}[theorem]{Remark}
\newtheorem{definition}[theorem]{Definition}
\newtheorem{example}[theorem]{Example}
\allowdisplaybreaks

\title{Qualitative Analysis of Some Dynamic Equations on Time Scales}
\author {Syed Abbas \\
School of Basic Sciences, \\
Indian Institute of Technology Mandi,\\
Kamand (H.P.) - 175 005, India.
\\
Email: abbas@iitmandi.ac.in}\maketitle
\author
\noindent {\bf Abstract} : In this paper, we establish the Picard-Lindel\"{o}f theorem and approximating results for dynamic  equations on time scale. We present a simple proof for the existence and uniqueness of the solution. The proof is produced by using convergence and Weierstrass M-test. Furthermore, we show that the Lispchitz condition is not necessary for uniqueness. The existence of $\epsilon-$approximate solution is established under suitable assumptions. Moreover, we study the approximate solution of the dynamic equation with delay by studying the solution of the corresponding dynamic equation with piecewise constant argument. We show that the exponential stability is preserved in such approximations.

\section{Introduction}

In this paper, we study the following dynamic equation
\begin{equation} \label{meq}
x^{\Delta}(t)=f(t,x(t)), \quad x(t_0)=x_0, \quad t \in [a,b]\cap \mathbb{T},
\end{equation}
where $\Delta$ is delta derivative, $\mathbb{T}$ is a time scale and $f$ is a rd-continuous function defined on $([a,b]\cap \mathbb{T}) \times \mathbb{R}$. We prove Picard-Lindel\"{o}f theorem for (\ref{meq}) using the assumption of rd-continuity and Lipschtiz and then establish existence of $\epsilon-$approximate solution. A solution of the above equation means a delta differentiable function which satisfies (\ref{meq}).

The theory of time scale calculus was introduced in the year 1988 by Stefan Hilger \cite{hilger}. This new theory unifies the calculus of the theory of difference equations with that of differential equations. It combines the analysis for integral and differential calculus with the calculus of finite differences. It gives a way to study hybrid discrete-continuous dynamical systems and has applications in any field that requires simultaneous modelling of discrete and continuous data. Hence, dynamic equations on a time scale have a potential for applications. In the population dynamics, the insect population can be better modelled using time scale calculus. The reason behind this is that they evolve continuously while in season, die out in winter while their eggs are incubating or dormant, and then hatch in a new season, giving rise to a non-overlapping population.

Lots of excellent books, monographs and research papers are available in this field majorally contributed by Martin Bohner, Allan Peterson, Lynn Erbe,  Ravi P. Agarwal, Samir H. Saker, Zhenlai Han,  Qi Ru Wang,  Youssef N. Raffoul and many more, we refer to \cite{14,aga1,16,7,boh3,boh4,17,18,cic,15,erbe1,erbe2,erbe3,9,8,10,raf1,sak1,sak2,12,1,3,11,2} and references therein. The books \cite{boh1,geo1} present a complete discussion on time scale calculus. A very nice survey article is written by Agarwal et.al. \cite{aga2}. For results on ordinary differential equations, we refer the books by Shair Ahmad and Rao, Arino and Conrad \cite{ahmad,4,6}, which are nice books for qualitative theory of ordinary differential equations. A linearlization method and topological classification for equations on time scale is discussed by Yonghui Xia et.al. in \cite{xia1,xia2}. In the paper \cite{cic}, M. Cicho\'{n} investigates a counterexample to Peano's Theorem on a time scale with only one right dense point.

In order to implement the continuous model for simulation purpose, it is essential to convert it into a discrete model. The resulting model will be a discrete equation which is easy to solve and implement. In this case we need to ensure that the discrete time model preserves the qualitative properties of the continuous time model. We can then use the continuous model without loss of functional similarity. It will also preserve the physical and biological reality that the continuous time model exhibits. In this work, we also consider a semilinear dynamic equation on time scale and study its approximate solutions. We show that the approximation preserves the exponential stability of the solution. For more details on equations with piecewise constant arguments, we refer to \cite{gopal1,moh1} and references therein.

The paper is organized as follows: In Section 2, we give necessary definitions and present some basic results, Section 3 is devoted to the existence of solution and the existence of $\epsilon-$approximate solution. Finally, in the last section, we discuss the approximate solution of a semilinear dynamic equation on time scale. We also establish the exponential stability of the solution.

\section{Preliminaries}

A time scale $\mathbb{T}$ is a nonempty closed subset of the real line $\mathbb{R}.$ The forward jump operator $\sigma(t)$ is defined by $\sigma(t)=\inf\{s\in \mathbb{T}: s>t\}.$ The right-dense point is defined be a point $t$ when $t < \sup \mathbb{T}$ and $\sigma(t)=t.$ It is called right scattered if $\sigma(t)>t.$ Similarly the backward jump operator $\rho(t)$ is defined by $\rho(t)=\sup\{s \in \mathbb{T}: s<t\}.$ So, a left-dense point is defined by the points such that $t>\inf \mathbb{T}$ and $\rho(t)=t.$ It is called left-scattered if $\rho(t)<t.$ We denote $\mathbb{T}^k=\mathbb{T}\setminus\{m\}$ or $\mathbb{T}_k=\mathbb{T}\setminus\{m\}$ if $\mathbb{T}$ has a left-scattered maximum or right-scattered minimum $m,$ respectively, otherwise $\mathbb{T}^k=\mathbb{T}_k=\mathbb{T}.$
\begin{definition}
A function $g:\mathbb{T} \rightarrow \mathbb{R}$ is called rd continuous if it is continuous at right-dense points of $T$ and its left-side limit exists at left-dense points.
\end{definition}
If $f$ is continuous at each right-dense point and each left-dense point, then $f$ is called continuous on $\mathbb{T}.$
\begin{definition}
For a function $g:\mathbb{T} \rightarrow \mathbb{R}, \ t \in \mathbb{T}^k,$ the delta derivative is defined as a function $f^{\Delta}(t),$ such that for each $\epsilon>0,$ there exists a neighbourhood $U$ of $t$ with the property
$$|f(\sigma(t)-f(s))-f^{\Delta}(t)(\sigma(t)-s)| \le \epsilon |\sigma(t)-s|,$$ for all $s \in U.$
\end{definition}

\begin{definition} The delta integral is defined as the antiderivative with respect to the delta derivative. If $F(t)$ has a continuous derivative $f(t)=F^{\Delta }(t),$ then
$\int _{r}^{t}f(s)\Delta s=F(t)-F(r).$
\end{definition}

\begin{remark} In the above definition, integral inequality does not holds for all time scales. For example in q-calculus the following relation is not correct: $\int_a^b D_q f(t) d_qt= f(b)-f(a)$. For more detail we refer the readers to \cite{boh1}.
\end{remark}

\begin{definition}
The set $Reg=\{p: \mathbb{T} \rightarrow R: p \in C_{rd}(\mathbb{T},\mathbb{R}), \ 1+p(t)\mu(t) \neq 0\}$ defines the set of regressive functions, where $\mu(t)=\sigma(t)-t.$
\end{definition}
For $p \in Reg,$ the exponential function $e_p(\cdot,t_0)$ is defined as the unique solution of the IVP $x^{\Delta}=p(t)x, \ x(t_0)=1.$ Further $Reg^+$ corresponds to $1+\mu(t)p(t)>0,$ and for $p \in Reg^+, \ e_p(\cdot,t_0)>0.$ Now we define the concept of $\epsilon-$approximate solution for the equation (\ref{meq}). Let $\Omega$ be any compact subset of $([a,b]\cap \mathbb{T}) \times \mathbb{R}.$
\begin{definition} A function $x(t)$ is called $\epsilon-$approximate solution of the equation (\ref{meq}) if
\begin{itemize}
\item[(i)] $(t,x(t)) \in \Omega, t \in [a,b]\cap \mathbb{T}.$
\item[(ii)] $x^{\Delta}(t) \in C^1_{rd}$ on $[a,b]\cap \mathbb{T}$ except possibly on a finite set $S$ (or a set of measure zero), where $x^{\Delta}$ may have simple discontinuities.
\item[(iii)] $\|x^{\Delta}(t)-f(t,x(t))\| \le \epsilon$ for each $t \in [a,b]\cap \mathbb{T}\backslash S.$
\end{itemize}
\end{definition}

\begin{lemma}(Weierstrass M-test)
Suppose that $\{\phi_n\}_{n \in N}$ is a sequence of real or complex-valued functions defined on a set A, and that there is a sequence of positive numbers $\{M_n\}$ satisfying
$\forall n\geq 1,\forall x\in A:\ |f_{n}(x)|\leq M_{n}, \sum _{{n=1}}^{{\infty }}M_{n}<\infty.$
Then the series $\sum _{{n=1}}^{{\infty }}f_{n}(x)$ converges absolutely and uniformly on A.
\end{lemma}

\begin{lemma}\cite{boh1,boh2}
Let $J$ be an arbitrary compact subset of $T$ and $\{x_n\}_{n \in N}$ is a sequence on $J$ such that $\{x_n\}_{n \in N}$ and $\{x_n^{\Delta}\}_{n \in N}$ are uniformly bounded on $J.$ Then, there exists a subsequence $\{x_{n_k}\}_{n_k \in N}$ which converges uniformly on $J.$
\end{lemma}

\begin{lemma}\cite{li1}
Let $\{x_n(t)\}_{n \in N}$ converges uniformly to $x(t)$ on $[a,b]\cap T$ and each $x_n(t)$ is continuous on $[a,b]\cap \mathbb{T}.$ Then the function $x(t)$ is continuous on  $[a,b]\cap \mathbb{T}$ and $$\lim_{n \rightarrow \infty} \int_{[a,b]\cap \mathbb{T}} x_n(s)\Delta s=\int_{[a,b]\cap \mathbb{T}} \lim_{n \rightarrow \infty}x_n(s)\Delta s=\int_{[a,b]\cap \mathbb{T}} x(s)\Delta s.$$
\end{lemma}

\begin{lemma}\cite{li1}
Let $\{x_n(t)\}_{n \in N}$ converges uniformly to $x(t)$ on $[a,b]\cap T$ and for each $n\in N,$ $x_n(t)$ has continuous delta derivative $x_n^{\Delta}(t).$ Moreover if $x_n^{\Delta}(t)$ converges uniformly to $y(t),$ then $x^{\Delta}(t)=y(t),$ and $x_n(t)$ converges to $x(t)$ uniformly on $[a,b]\cap \mathbb{T}.$
\end{lemma}

The following Gronwall-Bellman inequality for time scale is established in \cite{boh1}:
\begin{theorem}
Let $y \in C_{rd}(\mathbb{T},\mathbb{R})$ and $p\in C_{rd}(\mathbb{T},\mathbb{R})$ such that $p(t)\ge 0, \ 1+\mu(t)p(t)>0.$ Then
$$y(t)\le \alpha+\int^t_{t_0}y(s)p(s)\Delta s$$ implies
$y(t)\le \alpha e_p(t,t_0),$ for $t \in \mathbb{T}, \ t \ge t_0.$
\end{theorem}
Here the result still holds if $\alpha$ is replaced by any $f \in C_{rd}(\mathbb{T},\mathbb{R}).$ For more details see \cite{boh1}.
\section{Existence and Uniqueness}

The integral form of the equation (\ref{meq}) is given by
$$x(t)=x_0+\int_{[t_0,t]\cap \mathbb{T}} f(s,x(s))\Delta s.$$
Let us define $D=\{(t,x): t \in [t_0-a, t_0+a]\cap \mathbb{T}, \ |x-x_0| \le b\}$
\begin{theorem}(Picard-Lindel\"{o}f)\label{mthm}
Let $f$ be a function from $D$ to $\mathbb{R},$ rd-continuous in $t$ and Lipschitz in $x$ with Lipschitz constant $L$. Furthermore, if an initial point is not a right-dense point of $\mathbb{T}$, then there exists $h=\min\{a, \frac{b}{M}\},$ such that the problem has a unique solution in the interval $[t_0-h,t_0+h]\cap \mathbb{T},$ where $M=\max_{D}|f(t,x)|.$
\end{theorem}
\begin{proof} We define Picard's iterations
\begin{eqnarray}
\phi_0(t)&=& x_0 \nonumber \\
\phi_1(t)&=& x_0+\int_{[t_0,t]\cap \mathbb{T}} f(s,\phi_0(s))\Delta s \nonumber \\
 && \vdots \nonumber \\
\phi_n(t)&=& x_0+\int_{[t_0,t]\cap \mathbb{T}} f(s,\phi_{n-1}(s))\Delta s \nonumber \\
&& \cdot
\end{eqnarray}
It is enough to prove the result in the interval $[t_0,t_0+h]\cap \mathbb{T}.$ Now we prove the above theorem in four steps. \\
Step-1: (Well-posedness) First we need to show that $(t, \phi_n(t)) \in D,$ when $t\in [t_0,t_0+h]\cap \mathbb{T}.$ For $n=1, \ |\phi_1(t)-x_0| \le \int_{[t_0,t]\cap \mathbb{T}} |f(s,\phi_0(s))|\Delta s.$ But since $(t,\phi_0(t)) \in D,$ we get continuity of $f(t,\phi_0(t))$ on $D,$ which implies boundedness. Hence, $|\phi_1(t)-x_0| \le \int_{[t_0,t]\cap \mathbb{T}} M \Delta s \le M |t-t_0| \le Mh \le b.$ Now assume that for $n-1$ we have $|\phi_{n-1}(t)-x_0| \le b,$ we will prove $|\phi_n(t)-x_0| \le b.$ We have
$$|\phi_n(t)-x_0| \le \int_{[t_0,t]\cap \mathbb{T}} |f(s,\phi_{n-1}(s))|\Delta s.$$ Since $(t, \phi_{n-1}(t)) \in D,$ we have boundedness of $f(t, \phi_{n-1}(t)).$ Thus, we obtain
$$|\phi_n(t)-x_0| \le \int_{[t_0,t]\cap \mathbb{T}} M\Delta s \le M|t-t_0| \le Mh \le b.$$ Hence $(t, \phi_{n}(t)) \in D$ for each $t \in [t_0,t_0+h]\cap \mathbb{T}.$
\\
Step-2: (Estimate) We want to show here that $|\phi_n(t)-\phi_{n-1}(t)| \le M L^{n-1} \frac{(t-t_0)^n}{n!}=M L^{n-1}h_n(t,t_0).$ Again we use mathematical induction to establish this result. It is easy to verify for $n=1.$ For $n,$ let us compute,
\begin{eqnarray}
|\phi_n(t)-\phi_{n-1}(t)| &\le & \int_{[t_0,t]\cap \mathbb{T}} |f(s, \phi_{n-1}(s))-f(s,\phi_{n-2}(s))|\Delta s \nonumber \\
& \le & L \int_{[t_0,t]\cap \mathbb{T}} |\phi_{n-1}(s)-\phi_{n-2}(s)|\Delta s \nonumber \\
& \le & L \int_{[t_0,t]\cap \mathbb{T}} ML^{n-2}h_{n-1}(s,t_0)\Delta s \nonumber \\
& = & M L^{n-1}h_n(t,t_0).
\end{eqnarray}
Step-3: (Convergence) We can write $\phi_n(t)=x_0+\sum_{i=1}^n (\phi_i(t)-\phi_{i-1}(t)).$ Using the above estimate, we obtain $$|\phi_n(t)-\phi_{n-1}(t)| \le M L^{n-1}h_n(t,t_0) \le M L^{n-1}h_n(a,t_0) \le M L^{n-1}\frac{(a-t_0)^n}{n!},$$ hence the series $\sum_{n=1}^{\infty}(\phi_n(t)-\phi_{n-1}(t))$ converges uniformly and absolutely on $t \in [t_0,t_0+h]\cap \mathbb{T}.$ Hence $\phi_n(t)$ converges absolutely and uniformly to a function $\phi(t).$
Now taking limit on both sides of the integral equation (3.1), we obtain
$$\phi(t)=\lim_{n \rightarrow \infty}\phi_n(t)=x_0+\lim_{n \rightarrow \infty}\int_{[t_0,t]\cap \mathbb{T}} f(s,\phi_{n-1}(s))\Delta s.$$ Since $f$ is continuous as the convergence is uniform on $[t_0,t_0+h]\cap \mathbb{T},$ we can take the limit inside the integral, which gives $\phi(t)=x_0+\int_{[t_0,t]\cap \mathbb{T}} \lim_{n \rightarrow \infty}f(s,\phi_{n-1}(s))\Delta s.$ Since $f$ is Lipschitz, we get $\lim_{n \rightarrow \infty}f(t,\phi_{n}(t))=f(t,\phi(t)).$ Thus $\phi(t)$ satisfies $\phi(t)=x_0+\int_{[t_0,t]\cap \mathbb{T}} f(s,\phi(s))\Delta s.$ \\
Step-4: (Uniqueness) Let $\phi,\psi$ are two solutions. Define $\Phi(t)=|\phi(t)-\psi(t)|.$ It is easy to see here that $\Phi(t_0)=0.$ Now we have
\begin{eqnarray}
\Phi(t) &\le & \int_{[t_0,t]\cap \mathbb{T}} |f(s,\phi(s))-f(s,\psi(s))|\Delta s \le L \int_{[t_0,t]\cap \mathbb{T}} |\phi(s)-\psi(s)|\Delta s \le L \int_{[t_0,t]\cap \mathbb{T}} \Phi(s)\Delta s \nonumber \\ && \le L \int_{[t_0,\sigma(t)]\cap \mathbb{T}} \Phi(s)\Delta s.
\end{eqnarray}
The above inequality is equivalent to $$\Phi(t)-L\int_{[t_0,\sigma(t)]\cap \mathbb{T}} \Phi(s)\Delta s \le 0.$$
Hence using Gronwall-Bellman inequality, we get $\Phi=0$ on $[t_0,t_0+h]\cap \mathbb{T}.$ This completes the proof.
\end{proof}

To show that in many situations the Lipschitz condition is not necessary for uniqueness, let us consider
$$x^{\Delta}=f(x), \ x(t_0)=x_0,$$ where $f:\mathbb{R} \rightarrow \mathbb{R}$ is positive and continuous for all $x\in \mathbb{R}.$ To see the uniqueness, let us define
$F(t)=x_0+\int_{[t_0,t]\cap \mathbb{T}} \frac{1}{f(x(s))}\Delta s.$ Let us denote $\lim_{t \rightarrow \pm \infty} F(t)=l_{\pm}.$ Then $F$ is one-to-one. Moreover $F$ is continuously differentiable and $F^{\Delta}(t) >0.$ This observation implies the existence of a rd-continuously differentiable inverse from $G:(l_{-}, l_{+}) \rightarrow \mathbb{R}.$ The observation $G^{\Delta}(t)=\frac{1}{F^{\Delta}(F^{-1}(t))}=f(F^{-1}(t))=f(G(t))$ yields that $G$ is a solution of our equation. Now let us assume that $H$ is any other solution from $(t_0,x_0),$ which gives $\frac{H^{\Delta}(t)}{f(H(t))}=1=(F(H(t)))^{\Delta}$ for all $t\in [a,b]\cap \mathbb{T}.$ Hence, $F(H(t))=t+c,$ so $t_0+c=F(H(t_0))=F(x_0)=t_0,$ which implies $c=0.$ The above analysis yields $F(G(t))=F(H(t))=t,$ and since $F$ is one to one, we obtain $G=H.$ Hence, we achieve uniqueness without using the Lipschitz condition.

\begin{theorem}
Let $f(t,x)$ is rd-continuous on $B_0=\{(t,x): t \in [t_0,t_0+a]\cap \mathbb{T}, \ |x-x_0| \le b\}.$ Let $M=\max_{B_0}|f(t,x)|, h=\min\{a, \frac{b}{M}\}.$ Then for $\epsilon>0,$ there exists an $\epsilon-$approximate solution of (\ref{meq}) on $[t_0,t_0+h]\cap \mathbb{T}.$
\end{theorem}
\begin{proof} Since the function $f(t,x)$ is rd-continuous on a compact set $B_0,$ we can claim rd-uniform continuity on this set. Hence, for each $\epsilon>0,$ there exists a $\delta>0$ such that
$$|f(t,x)-f(s,y)| \le \epsilon,$$ whenever
$|t-s| \le \delta$ and $|x-x_0| \le \delta,$ for $(t,x), (s,y) \in B_0.$
The interval $[t_0, t_0+h] \cap \mathbb{T}$ can be divided into $n$ subintervals
$$t_0 <t_1<t_2<\cdots <t_n=t_0+h,$$ where each subinterval is taken as intersection with the time scale $\mathbb{T}.$ It is easy to note that $\max|t_k-t_{k-1}| \le \min\{\delta, \frac{\delta}{M}\}.$ Now we define a function $x(t)$ such that
$$x(\sigma(t))=
\begin{cases}
x_0, \quad t=t_0 \\
x(t_{k-1})+(\sigma(t)-t_{k-1})f(t_{k-1},x(t_{k-1})), \quad t \in (t_{k-1}, t_k]\cap \mathbb{T},
\end{cases}
k=1,2,\cdots, n.$$
It is easy to verify that $x(t)$ is rd-continuously differentiable function on the interval $[t_0,t_0+h]\cap \mathbb{T}.$ Further, $$|x(t)-x(s)| \le M |t-s|, \quad t,s \in [t_0, t_0+h]\cap \mathbb{T}.$$ Now for $t\in (t_{k-1},t_k)\cap \mathbb{T},$ it follows that $|x(t)-x(t_{k-1})| \le \delta.$ Combining all the above arguments, we have
$$|x^{\Delta}(t)-f(t,x(t))|=|f(t_{k-1},x(t_{k-1}))-f(t,x(t))| \le \epsilon,$$
which implies the existence of $\epsilon-$approximate solution. This completes the proof.
\end{proof}

\begin{remark}
For $x \in \mathbb{R}^n,$ the same proof will work by replacing modulus by Euclidean norm.
\end{remark}
\begin{example}
Consider the following equation on $T,$
$x^{\Delta}=\sqrt{x}, \ x(0)=0.$ If $\mathbb{T}=\mathbb{R},$ then equation becomes $x^{\prime}=\sqrt{x}, \ x(0)=0,$ which has more than one solution, namely $x(t)=0$ and $x(t)=\frac{t^2}{4}$ and combination of theses. Further when $\mathbb{T}=\mathbb{Z},$ the equation becomes $\Delta x=\sqrt{x}, \ x(0)=0,$ we have $x=0$ is solution. This suggests that Lipschitz condition is required for uniqueness. At the same time if we drop the Lipschitz condition, then we have at least one solution. Similarly for $\mathbb{T}=\mathbb{N},$ Cantor set etc such that $\mu>0$, we have unique solution. For the time scale $\mathbb{T}=\{q^n: n \in \mathbb{N}\}\cup \{0\}$ such that $q>0,$ then $x^{\Delta}=\sqrt{x}, \ x(0)=0$ has unique solution $x=0$. So, we conclude that the condition of the Picard-Lindel\"{o}f is sufficient but not necessary for the existence of unique solution on any time scale $\mathbb{T}.$ Also the nature changes when we move from $\mathbb{R}$ to any other time scale.
\end{example}
\begin{remark}
One can also show the existence of solution of the dynamic equation on time scale with delay. Let us consider
\begin{eqnarray} \label{dmeq}
x^{\Delta}(t) &&= f(t,x(t-\tau)), \quad t \in [a, b]\cap \mathbb{T},
\nonumber \\
x(t)&&= h(t), \quad t \in [a-\tau, a)\cap \mathbb{T}.
\end{eqnarray}
The method is the usual method of steps. In the interval $[a,a+\tau]\cap \mathbb{T},$ one can use Theorem \ref{mthm} to show the existence of solutions and then proceed with other intervals. When $t \in [a,a+\tau]\cap \mathbb{T},$ then $t-\tau \in [a-\tau,a]\cap \mathbb{T}.$ In this case, our equation is $x^{\Delta}(t)= f(t,h(t-\tau)).$ The last equation can be solved easily. Now after getting the solution in the time scale interval $[a,a+\tau]\cap \mathbb{T},$ one can solve it in the interval $[a+\tau,a+2\tau]\cap \mathbb{T}$ and so on.
\end{remark}

\section{Approximation}
In this section, we suppose that $\sup\mathbb{T}=\infty.$ Now, let us consider the the following equations
$$x^{\Delta}(t)=A(t)x(t),$$ where $A(t)$ is $n \times n$ matrix which is continuous. Further, consider the following dynamic equation with delay (DED)
$$y^{\Delta}(t)=A(t)y(t)+f(t,y(t-\tau)), \ \tau>0,$$ with history $\eta(t)$ in the interval $[-\tau,0]\cap \mathbb{T}.$ The corresponding differential equation with piecewise constant argument (DEPCA) is given by
$$z_h^{\Delta}(t)=A(t)z_h(t)+f(t,z_h(\gamma_h(t-\tau))),$$ with history $z_h(nh)=\eta(nh)$ for $n=-k,\cdots,0.$ The step $h=\frac{\tau}{k}$ and $k \ge 1$ is an integer and $\gamma_h(t-\tau)=\Big[\frac{s}{h}-\Big[\frac{\tau}{h}\Big]\Big]h,$ where $[\cdot]$ is greatest integer function.
The solution of the above problem is a function $z_h$ which is continuous on $\mathbb{T}^+=\mathbb{T} \cap [0,\infty)$ and $z_h^{\Delta}(t)$ exists for each $t\in \mathbb{T}^+$ with possible exception on $kh,$ where one sided limit exist and it satisfies DEPCA on each interval $I_k:[kh, (k+1)h]\cap \mathbb{T}.$

Our aim here is to compare the solutions of DED and DEPCA. Since as $h \rightarrow 0, \ [t]_h \rightarrow t,$ uniformly on $\mathbb{T},$ for $0<h\le h_0,$ it is expected that the solutions of both equations show similar qualitative properties. In the interval $I_i,$ the DEPCA can be written as
$$z_h^{\Delta}(t)=A(t)z_h(t)+f(t,z_h(h(i-k))),$$ with the same initial condition. Now we can use the variation of parameter formula (for details, we refer to \cite{boh1,lak}) to express the solution, which is
$$z_h(t)=e_{A(t)}(t,ih)z_h(ih)+\int_{[ih,t]\cap \mathbb{T}}e_{A(t)}(t,\sigma(s))f(s,z_h((i-k)h))\Delta s.$$
Now taking $t \rightarrow (i+1)h,$ we get
$$z_h((i+1)h)=e_{A(t)}((i+1)h,ih)z_h(ih)+\int_{[ih,(i+1)h]\cap \mathbb{T}}e_{A(t)}((i+1)h,\sigma(s))f(s,z_h((i-k)h))\Delta s.$$ Now we define a sequence $a_h(i)=z_h(ih).$ One can easily check that it satisfies the following difference equation
\begin{eqnarray} \label{mdiff}
a_h(n+1)&=& e_{A(t)}((i+1)h,ih)a_h(n)+\int_{[ih,(i+1)h]\cap \mathbb{T}}e_{A(t)}((n+1)h,\sigma(s))f(s,a_h((n-k)))\Delta s, \nonumber \\ a_h(n)&=& \phi(nh),
\end{eqnarray}
 for $n=0,1,\cdots, -\lambda \le -nh \le 0.$
So, we have obtained an approximation of the original problem DED. Now we can compute few values of $a_h$ for $n=0,1,2.$
$$a_h(0)=\phi(0), \ a_h(1)=e_{A(t)}(h,0)a_h(0)+\int_{[0,h]\cap \mathbb{T}}e_{A(t)}(h,\sigma(s))f(s,a_h((-k)))\Delta s,$$ and
\begin{eqnarray}
a_h(2)&=& e_{A(t)}(2h,h)a_h(1)+\int_{[h,2h]\cap \mathbb{T}}e_{A(t)}(2h,\sigma(s))f(s,a_h((1-k)))\Delta s \nonumber \\
&=& e_{A(t)}(2h,h)\Big(e_{A(t)}(h,0)a_h(0)+\int_{[0,h]\cap \mathbb{T}}e_{A(t)}(h,\sigma(s))f(s,a_h((0-k)))\Delta s\Big)\nonumber \\ &&+\int_{[h,2h]\cap T}e_{A(t)}(2h,\sigma(s))f(s,a_h((1-k)))\Delta s \nonumber \\
&=& e_{A(t)}(2h,0)\phi(0)+\int_{[0,h]\cap \mathbb{T}}e_{A(t)}(2h,\sigma(s))f(s,a_h((0-k)))\Delta s\nonumber \\ &&+\int_{[h,2h]\cap \mathbb{T}}e_{A(t)}(2h,\sigma(s))f(s,a_h((1-k)))\Delta s.
\end{eqnarray}
Hence, we obtain the following relation
$$a_h(n)=e_{A(t)}(nh,0)\phi(0)+\sum_{i=0}^{n-1} \int_{[ih,(i+1)h]\cap \mathbb{T}}e_{A(t)}(nh,\sigma(s))f(s,a_h((i-k)))\Delta s.$$ Using the above expressions, we get the following representation of $z_h(t):$
$$z_h(t)=e_{A(t)}(t,nh)a_h(n)+\int_{[nh,t]\cap \mathbb{T}}e_{A(t)}(t,\sigma(s))f(s,a_h((n-k)))\Delta s.$$ Substituting the value of $a_h(n),$ we obtain
\begin{eqnarray}
z_h(t)&=& e_{A(t)}(t,0)\phi(0)+\sum_{i=0}^{n-1}\int_{[ih,(i+1)h]\cap \mathbb{T}}e_{A(t)}(t,\sigma(s))f(s,a_h((i-k)))\Delta s\nonumber \\ &+& \int_{[nh,t]\cap \mathbb{T}}e_{A(t)}(t,\sigma(s))f(s,a_h((n-k)))\Delta s,
\end{eqnarray}
for $nh \le t \le (n+1)h.$

Using the above observation, we can state the following theorem.
\begin{theorem}
Under the assumptions of Theorem $3.1$, the DEPCA with given initial condition has solution of the form
\begin{eqnarray}
z_h(t)&=& e_{A(t)}(t,0)\phi(0)+\sum_{i=0}^{n-1}\int_{[ih,(i+1)h]\cap \mathbb{T}}e_{A(t)}(t,\sigma(s))f(s,a_h((i-k)))\Delta s\nonumber \\ &&+\int_{[nh,t]\cap \mathbb{T}}e_{A(t)}(t,\sigma(s))f(s,a_h((n-k)))\Delta s,
\end{eqnarray}
for $t \in \mathbb{T}^+$ and the sequence $a_h(\cdot)$ satisfies the nonlinear difference equation (\ref{mdiff}).
\end{theorem}
We note the importance of uniqueness of the solution. Otherwise, the approximate method with piecewise constant argument may not work.

Now we establish a connection between exponential stability of the solution of a linear system with the approximation of the solution of DED for large $t.$ This implies that if the zero solution of the linear system is exponentially stable, then the approximate solution tends to solution as $t$ approaches infinity. Here we consider the real delay $\tau$ and $\mathbb{T}^+=\mathbb{T} \cap \mathbb{R}^+.$
Let us assume the following:
\begin{itemize}
\item[A1.] $A(t)$ is regressive and rd-continuous matrix.
\item[A2.]  The fundamental matrix $e_{A(t)}(t,s)$ satisfies $|e_{A(t)}(t,s)| \le M e^{-\lambda (t-s)},$ for $t \ge s$ and $M, \lambda$ are positive constants.
\item[A3.] $f: \mathbb{T}^+ \times \mathbb{R}^n \rightarrow \mathbb{R}^n$ satisfies $f(t,0)=0$ and
$$\|f(t,x)-f(t,y)\| \le L \|x-y\|.$$
\end{itemize}
\begin{lemma}
Under the assumptions $A1-A3,$ if the zero solution of DED is exponentially stable, then the solution satisfies
$$\|y(t-\tau)-y(\gamma_h(t-\tau))\| \le M^* e^{-\lambda s} \ s \ge 2\tau,$$ where
$M^*=Me^{2\lambda \tau} \int_{[\gamma_h(t-\tau), t-\tau]\cap \mathbb{T}} (\|A(s)\|+L e^{\lambda \tau}) \Delta s,$ which tends to zero as $h$ tends to zero.
\end{lemma}

\begin{proof} We have
\begin{eqnarray}
\|y(t-\tau)-y(\gamma_h(t-\tau))\| &=& \|\int_{[\gamma_h(t-\tau), t-\tau]\cap \mathbb{T}}y^{\Delta}(s) \Delta s\|
\nonumber \\ &=& \|\int_{[\gamma_h(t-\tau), t-\tau]\cap \mathbb{T}} (A(s)y(s)+f(s,y(s-\tau))) \Delta s\|
\nonumber \\ & \le & \int_{[\gamma_h(t-\tau), t-\tau]\cap \mathbb{T}} \|A(s)y(s)+f(s,y(s-\tau))\| \Delta s
\nonumber \\ & \le & \int_{[\gamma_h(t-\tau), t-\tau]\cap \mathbb{T}} \|A(s)y(s)+f(s,y(s-\tau))-f(s,0)+f(s,0)\| \Delta s
\nonumber \\ & \le & \int_{[\gamma_h(t-\tau), t-\tau]\cap \mathbb{T}} (\|A(s)\|\|y(s)\|+L\|y(s-\tau)\|+\|f(s,0)\|) \Delta s
\nonumber \\ & = & \int_{[\gamma_h(t-\tau), t-\tau]\cap \mathbb{T}} (\|A(s)\|\|y(s)\|+L\|y(s-\tau)\|) \Delta s.
\end{eqnarray}
Since solutions are exponentially stable, we have $\|y(s)\| \le Me^{-\lambda s},$ and from here
\begin{eqnarray}
\|y(t-\tau)-y(\gamma_h(t-\tau))\| & \le & \int_{[\gamma_h(t-\tau), t-\tau]\cap \mathbb{T}} (\|A(s)\|Me^{-\lambda s}+LM e^{-\lambda (s-\tau)}) \Delta s
\nonumber \\ & \le & Me^{-\lambda (t-\tau -h)} \int_{[\gamma_h(t-\tau), t-\tau]\cap \mathbb{T}} (\|A(s)\|+L e^{\lambda \tau}) \Delta s
\nonumber \\ & \le & M e^{-\lambda t} e^{2\lambda \tau} \int_{[\gamma_h(t-\tau), t-\tau]\cap \mathbb{T}} (\|A(s)\|+L e^{\lambda \tau}) \Delta s. \nonumber
\end{eqnarray}
It is easy to see that the measure of the interval $[\gamma_h(t-\tau), t-\tau]\cap \mathbb{T}$ is $h$ and the function $A(s)$ is regressive and rd-continuous. Therefore
$$e^{2\lambda \tau} \int_{[\gamma_h(t-\tau), t-\tau]\cap \mathbb{T}} (\|A(s)\|+L e^{\lambda \tau}) \Delta s \rightarrow 0, \ h \rightarrow 0,$$ which implies our result.
\end{proof}
Now we can use the above lemma to derive the result for approximate solution $z_h(t).$
\begin{theorem}
Under the assumptions $A1-A3,$ if the zero solution of DED is exponentially stable and following condition holds $e^{\lambda \tau} L \le \lambda,$ we have that for every $h \in (0,h_0]$ the following holds
$$y(\phi)(t)-z_h(t) \le \Big(v(0)+\int_{[2\tau, t]\cap \mathbb{T}}LM^* \Delta s+w(y,h,2\tau)\int_{[0,2\tau]\cap \mathbb{T}}e^{\lambda s}L \Delta s\Big)e_{-\lambda_0}(1,0),$$ where
$w(y,h,2\tau)=\max\{\|y(s-\tau)-y(\gamma_h(s-\tau))\|: \ 0 \le s \le 2\tau\}$ and $\lambda_0=\lambda-e^{\lambda(2h+\tau)}L.$
\end{theorem}
\begin{proof} We can claim that there exists $h_0$ such that $e^{\lambda (2h_0+\tau)}L=\lambda.$ Let us define $E_h(t)=y(t)-z_h(t)$ for $t \in I_n.$ Taking $\Delta-$derivative of $E_h,$ we obtain
\begin{eqnarray}
E_h^{\Delta}(t)&=& y^{\Delta}(t)-z_h^{\Delta}(t) \nonumber \\
&=& A(t)E_h(t)+f(t,y(t-\tau))-f(t,z_h(h(n-k))) \nonumber \\
&=& A(t)E_h(t)+f(t,y(t-\tau))-f(t,z_h(h(n-k)))+f(t,y(h(n-k)))-f(t,y(h(n-k))). \nonumber
\end{eqnarray}
The last equation is a dynamic equation in $E_h$ and its solution can be represented by
$$E_h(t)=e_{A(t)}(t,0)E_h(0)+\int_{[0,t]\cap \mathbb{T}} e_{A(t)}(t,\sigma(s))F(s,y(s-\tau),z_h(s))\Delta s.$$
Taking norm on both sides of the last equation and using our assumptions, we obtain
\begin{eqnarray}
\|E_h(t)\|&\le & e^{-\lambda t}E_h(0)+\int_{[0,t]\cap \mathbb{T}}e^{-\lambda (t-\sigma(s))}L\|E_h(\gamma_h(s-\tau))\|\Delta s \nonumber \\
&+& \int_{[0,t]\cap \mathbb{T}}e^{-\lambda (t-\sigma(s))}L\|y(s-\tau)-y(\gamma_h(s-\tau))\|\Delta s. \nonumber
\end{eqnarray}
Using the Lemma 4.2, we get
\begin{eqnarray}
\|E_h(t)\|&\le & e^{-\lambda t}E_h(0)+\int_{[0,t]\cap \mathbb{T}}e^{-\lambda (t-\sigma(s))}L\|E_h(\gamma_h(s-\tau))\|\Delta s \nonumber \\
&+& \int_{[0,2\tau]\cap \mathbb{T}}e^{-\lambda (t-\sigma(s))}L\|y(s-\tau)-y(\gamma_h(s-\tau))\|\Delta s \nonumber \\ &+& \int_{[2\tau,t]\cap \mathbb{T}}e^{-\lambda t}LM^* \Delta s. \nonumber
\end{eqnarray}
The above relation implies
\begin{eqnarray}
e^{\lambda t}\|E_h(t)\|&\le & E_h(0)+L\int_{[0,t]\cap \mathbb{T}}e^{\lambda (\sigma(s)-\gamma_h(s-\tau))} e^{\lambda \gamma_h(s-\tau)}\|E_h(\gamma_h(s-\tau))\|\Delta s \nonumber \\
&+& L\int_{[0,2\tau]\cap \mathbb{T}}e^{\lambda \sigma(s)}\|y(s-\tau)-y(\gamma_h(s-\tau))\|\Delta s+\int_{[2\tau,t]\cap \mathbb{T}}LM^* \Delta s. \nonumber
\end{eqnarray}
Now let us define $v(t)=\sup_{s \in [-r,t]}e^{\lambda s}\|E_h(s)\|.$ Then we get
\begin{eqnarray}
v(t)&\le & v(0)+L\int_{[0,t]\cap \mathbb{T}}e^{\lambda (\sigma(s)-\gamma_h(s-\tau))} v(s)\Delta s \nonumber \\
&+& L\int_{[0,2\tau]\cap \mathbb{T}}e^{\lambda \sigma(s)}\|y(s-\tau)-y(\gamma_h(s-\tau))\|\Delta s+\int_{[2\tau,t]\cap \mathbb{T}}LM^* \Delta s. \nonumber
\end{eqnarray}
Applying the Gronwall-Bellman inequality, we obtain
\begin{eqnarray}
v(t)\le \Big(v(0)+w(y,h,2\tau)\int_{[0,2\tau]\cap \mathbb{T}}e^{\lambda \sigma(s)}\Delta s+\int_{[2\tau,t]\cap \mathbb{T}}LM^* \Delta s\Big)e_p(t,0),
\end{eqnarray}
where $p(s)=L e^{\lambda (\sigma(s)-\gamma_h(s-\tau))}.$

We can easily observe that $\sigma(s)-\gamma_h(s-\tau)=\sigma(s)-([s]_h-[\tau]_h) \le h+[\tau]_h \le 2h+\tau.$ Hence, we finally get
\begin{eqnarray}
v(t)\le \Big(v(0)+w(y,h,2\tau)\int_{[0,2\tau]\cap \mathbb{T}}e^{\lambda \sigma(s)}\Delta s+\int_{[2\tau,t]\cap \mathbb{T}}LM^* \Delta s\Big)e_q(t,0),
\end{eqnarray}
where $q=e^{\lambda (2h+\tau)}L-\lambda$.

From here, it follows with $q=-\lambda_0.$ Moreover, since it is assumed that the zero solution of the equation DED is exponentially stable and $\lambda_0>0$, the above relation implies $y(\phi)(t)$ tends to zero as $t$ tends to infinity. Hence, the solution $z_h(\phi)(t)$ tends to zero when $t \rightarrow \infty.$ Therefore the system DEPCA is exponentially stable. This completes the proof.
\end{proof}

\vskip .2cm
\textbf{Acknowledgement:} We are thankful to the anonymous reviewers and editor for their constructive comments and suggestions, which helped us to improve the manuscript considerably.

\end{document}